\documentclass[11pt, a4paper]{amsart}
\usepackage{blindtext}
\usepackage{graphicx}
\usepackage{amsfonts}
\usepackage{tikz-cd}
\usepackage{mathtools}
\usepackage{amscd,amsmath}
\usepackage{amssymb}
\usepackage{xfrac}
\usepackage{amscd,amsmath}
\usepackage{amssymb}
\usepackage{fullpage}
\usepackage{setspace}
\newtheorem{theorem}{Theorem}[section]
\newtheorem{lemma}[theorem]{Lemma}

\newtheorem{proposition}[theorem]{Proposition}
\newtheorem{definition}[theorem]{Definition}
\newtheorem{example}[theorem]{Example}

\theoremstyle{remark}
\newtheorem{remark}{Remark}[section]
\newtheorem*{rem*}{Remark}

\newcommand{\comment}[1]{}
\numberwithin{equation}{section}

\begin{document}

\title{Theory of Extensions of Multiplicative Lie Algebras}
\author{Mani Shankar Pandey$^1$ and Sumit Kumar Upadhyay$^2$\vspace{.4cm}
}
\address{$^{1,2}$Department of Applied Sciences,\\ Indian Institute of Information Technology, Allahabad}
%\address{$^{2}$ HRI, Allahabad\\}
\email{\tiny{$^1$manishankarpandey4@gmail.com, $^2$upadhyaysumit365@gmail.com}}.\\
\begin{abstract}
In this paper, we attempt to develop the Schreier theory for two special types extensions of multiplicative Lie algebras.
\end{abstract}
\maketitle
\section{Introduction}
In 1993, G. J. Ellis (\cite{GJ})  asked the following question: is any universal commutator identity a consequence of all five well known universal commutator identities? He gave a partial affirmative answer to this question by using homological techniques and also introduced an interesting algebraic object called ``multiplicative Lie algebras" which is a generalisation of groups as well as Lie algebras. In \cite{FP} F. Point and P. Wantiez  introduced concept of nilpotency for multiplicative Lie algebras and proved many nilpotency results which are known for groups and Lie algebras. To see more structural properties of multiplicative Lie algebra, in \cite{AGNM}  A Bak. et al.  constructed and studied two homology theories for multiplicative Lie algebras which are already known for groups and Lie algebras. They also introdued central extensions of multiplicative Lie algebra. Furthermore, in 2015 G. Donadze et al.  \cite{GNM} gave the notion of non abelian tensor product of multiplicative Lie algebras and also established a relationship between the second homology and the non abelian exterior square of a multiplicative Lie algebra. In 2018, authors \cite{GNMA} proved some algebraic structural results for non abelian tensor product of multiplicative Lie algebras. They also obtained six term exact sequence in homology of multiplicative Lie algebras and proved a new version of Stallings theorem. 

One of the problem in the theory of extensions is ``to determine all multiplicative Lie algebras $G$ (up to isomorphism) having $H$ as an ideal such that $G/H$ is isomorphic to $K$, for any two multiplicative Lie algebras $H$ and $K$" In other words ``what are the classifications of all 2-fold extensions (upto equivalence) of $H$ by $K$"? In 2019, R. Lal and S. K. Upadhyay \cite{RLS} developed the Schreiers extension theory for 2-fold extension for a special case termed as ``central extension" in an attempt to derive the Schur-Hopf formula for the multiplicative Lie algebras. 
In this paper, we attempt to develop the Schreiers extension theory for 2-fold extension of an abelian group with trivial multiplicative Lie algebra structure by an arbitrary multiplicative Lie algebra analogus to group theory in two different cases, namely center extension and Lie center extension (it is known as Central extension in \cite{AGNM}). Moreover, there is an interesting problem is to develop the Schreiers extensions theory for n-fold extensions of multiplicative Lie algebras. Most of the definitions and terminology are taken from \cite{RL} $\&$ \cite{RLS}. Now we will give some basic definitions in the same manner to the extension theory of groups.
 
%A multiplicative Lie algebra is a triplet $(G,\cdot,*)$, where $(G,\cdot)$ is a group and $*$ is a binary operation on $G$ which satisfies all five well known commutator identities. This structure was introduced by G. J. Ellis \cite{GJ}. It is clear from the definition that a multiplicative Lie algebra is generalization of groups as well as Lie algebras. 

\begin{definition} 
\begin{enumerate}
\item A short exact sequence 
\begin{center}
\begin{tikzcd}
  E(H, K) \equiv 1\arrow{r} & H\arrow{r}{\alpha} & G \arrow{r}{\beta} & K\arrow{r} & 1
\end{tikzcd}
\end{center}
of multiplicative Lie algebras is called an extension of $H$ by $K$. Here it is clear that $\alpha$ is injective, so we can think $H$ as an ideal of $G$ and $\alpha$ as an inclusion map $i$.
\item A morphism from an extension 
\begin{center}
\begin{tikzcd}
  E(H, K) \equiv 1\arrow{r} & H\arrow{r}{\alpha} & G \arrow{r}{\beta} & K\arrow{r} & 1
\end{tikzcd}
\end{center} to an extension 
\begin{center}
\begin{tikzcd}
  E(H', K') \equiv 1\arrow{r} & H'\arrow{r}{\alpha'} & G' \arrow{r}{\beta'} & K'\arrow{r} & 1
\end{tikzcd}
\end{center}
is a triple $(\lambda, \mu, \nu)$, where $\lambda: H \longrightarrow H'$, $\mu: G \longrightarrow G'$ and $\nu : K \longrightarrow K'$ are multiplicative Lie algebra homomorphisms such that the diagram 
\begin{center}

\begin{tikzcd}
   E(H, K) \equiv 1\arrow{r} & H \arrow{r}{\alpha}\arrow{d}{\lambda} & G \arrow{r}{\beta}\arrow{d}{\mu} & K\arrow{r}\arrow{d}{\nu} & 1 \\
  E(H', K') \equiv 1\arrow{r} & H'\arrow{r}{\alpha'} & G'\arrow{r}{\beta'} & K'\arrow{r} & 1
\end{tikzcd}
\end{center} 
is commutative. 
%Further these extensions $E_1$ and $E_2$ are said to be equivalent if the homomorphisms $\lambda, \ \mu \ \text{and} \ \nu$ are isomorphism.
%\end{enumerate}
%\end{definition}
%
%\begin{remark}
%It can be easily observed that, there is no loss of generality in restricting the concept of equivalence on the class $E(H,K)$ of all center extensions of $H$ by $K$, where $E(H, K)$ denote the set of all equivalence classes of center extensions of $H$ by $K$
%\end{remark}
%This prompts us to following definition: 
%\begin{definition}
\item Two extensions $E_1(H, K)$ and $E_2(H, K)$ of $H$ by $K$ are equivalent if there is an isomorphism $\phi :G_1 \longrightarrow G_2$ such that the following
diagram
\begin{center}
\begin{tikzcd}
   E_1(H, K) \equiv 1\arrow{r} & H \arrow{r}{i_1}\arrow{d}{I_H} & G_1 \arrow{r}{\beta_1}\arrow{d}{\phi} & K\arrow{r}\arrow{d}{I_K} & 1 \\
   E_2(H, K) \equiv 1\arrow{r} & H\arrow{r}{i_2} & G_2\arrow{r}{\beta_2} & K\arrow{r} & 1
\end{tikzcd}
\end{center}
is commutative.
\end{enumerate}
\end{definition}
Throughout this paper, $H$ denotes an abelian group with trivial multiplicative Lie algebra structure and $K$ denotes an arbitrary multiplicative Lie algebra.
\section{Schreier Theory for Center extension}
In this section, we discuss the theory of center extension. 
\begin{definition}
An extension \begin{tikzcd}
  E(H, K) \equiv 1\arrow{r} & H\arrow{r}{i} & G \arrow{r}{\beta} & K\arrow{r} & 1
\end{tikzcd} of an abelian group $H$ with trivial multiplicative Lie algebra structure by a multiplicative Lie algebra $K$ is called a \textbf{center extension} if $H$ is contained in the center of $G$.
\end{definition}
\begin{remark} Let $H$ be an abelian group and $End(H)$ be the set of all group endomorphism on $H$. Then $End(H)$ is a multiplicative Lie algebra with the operations $\cdot$ and $*$, where $\cdot$ and $*$ are defined by $(F\cdot G)(h) = F(h)G(h)$ and $(F*G)(h) = F(G(h))G(F(h^{-1}))$ respectively.
\end{remark}
Let 
\begin{tikzcd}
  E(H, K) \equiv 1\arrow{r} & H\arrow{r}{i} & G \arrow{r}{\beta} & K\arrow{r} & 1
\end{tikzcd} be a center extension of $H$ by $K$ and $t : K \longrightarrow G$ be a section of $E$. Then  it is easy to see that every element of group $G$ can be uniquely expressed in the form $ht(x)$, for some $h \in H$ and $x \in K$. 

The group operation $``\cdot " $ (for details see \cite{RLS}) in $G$ is given by $$\mathbf{(ht(x))\cdot (kt(y))=hkf^t(x,y)t(xy)},\hspace{2cm} ...(1)$$ where $f^t : K \times K \rightarrow H$ is a map satisfying 
\begin{equation}\label{1}
\mathbf{f^t(1,x)=f^t(x,1)=1 ~\text{and}~ f^t(x,y)f^t(xy,z) = f^t(y,z)f^t(x,yz)}
\end{equation}

Further $\beta(t(x)*h)=\beta(t(x))*\beta(h)=x*1=1$, $t(x)*h \in \ker(\beta) = H$. Thus for all $x \in K$, there is a map $\Gamma^t_x : H\rightarrow H$ defined by $ \Gamma^t_x(h)=t(x)*h$. 

Also for $h,k \in H$, $\Gamma^t_x(hk)=t(x)*(hk)= (t(x)*h)^h(t(x)*k) =(t(x)*h)(t(x)*k)$ \\
$\Rightarrow \Gamma^t_x(hk)=\Gamma^{t}_x(h)\Gamma^t_x(k)$. This shows that for all $x \in K$,  $\Gamma^t_x$ is a group endomorphism on $H$. 

Now the multiplicative Lie algebra structure $``*"$ in $G$ is given as follows: 

Since $\beta (t(x)*t(y)) = \beta(t(x))*\beta(t(y))= x*y = \beta(t(x*y))$, $(t(x)*t(y))\cdot(t(x*y))^{-1} \\ \in \ker(\beta)$. Hence there exists a map $h^t : K \times K \rightarrow H$ such that $$t(x)*t(y)=h^t(x,y)t(x*y).$$

Since $t(x)*t(x) = t(x)*1 = 1*t(x) = 1$, we have 
\begin{equation}\label{2}
\mathbf{h^t(x,x)=h^t(x,1)=h^t(1,x)=1}
\end{equation}
Further,  
 $$h*(kt(y))=(h*k)^k(h*t(y))=h*t(y)=\Gamma^t_y(h^{-1})$$
and
$$t(x)*(kt(y))=(t(x)*k)^k(t(x)*t(y))= \Gamma^t_x(k)h^t(x,y)t(x*y)$$

Therefore 
$(ht(x))*(kt(y))=^h(t(x)*kt(y))(h*kt(y))=(t(x)*(kt(y)))(h*(kt(y)))=\Gamma^t_x(k)\Gamma^t_y(h^{-1})h^t(x,y)t(x*y)$ (since $H \subseteq Z(G)$). Thus the multiplicative Lie algebra structure $``*"$ in $G$ is given by 
\begin{center}
$\mathbf{(ht(x))*(kt(y))=\Gamma^t_x(k)\Gamma^t_y(h^{-1})h^t(x,y)t(x*y)}\hspace{2cm} ...(2)$
\end{center}
\begin{lemma}
The map $\Gamma^t_{E}:K\longrightarrow End(H)$ defined by $\Gamma^t_{E}(x) = \Gamma^t_x$ is a multiplicative Lie algebra homomorphism.
\end{lemma}
\begin{proof}
Let $x,y \in K$ and $h \in H$. Then 
 $\Gamma^t_{xy}(h)=t(xy)*h 
 =(f^t(x,y)^{-1}t(x)t(y))*h 
 =^{f^t(x,y)^{-1}}((t(x)t(y))*h)$ $(f^t(x,y)^{-1}*h)$ 
 $=(t(x)t(y))*h
 =^{t(x)}(t(y)*h)\cdot (t(x)*h) \\
 =(t(y)*h)\cdot (t(x)*h)$
 $=\Gamma^t_x(h)\Gamma^t_y(h))$
 $ \Rightarrow \Gamma^t_{xy}=\Gamma^t_x \Gamma^t_y$ $\Rightarrow\Gamma^t_{E}(xy) = \Gamma^t_{E}(x)\Gamma^t_{E}(y)$. Thus $\Gamma^t_{E}$ is a group homomorphism.
 
Now $\Gamma^t_{x*y}(h)=t(x*y)*h=(h^t(x,y)^{-1}(t(x)*t(y)))*h=^{h^t(x,y)^{-1}}((t(x)*t(y))*h).(h^t(x,y)^{-1}*h)=(t(x)*t(y))*h$. By Jacobi identity,

$((t(x)*t(y))*^{t(y)}h)((t(y)*h)*^ht(x))(h*t(x))*^{t(x)}t(y))=1$ $\Rightarrow ((t(x)*t(y))*h)(\Gamma^t_y(h)*t(x)) (\Gamma^t_x(h^{-1})*(f^t(x,y)t(xy)t(x)^{-1})=1$ $ \Rightarrow (\Gamma^t_{x*y}(h) \Gamma^t_x(\Gamma^t_y(h^{-1}))) (\Gamma^t_x(h^{-1})*(f^t(x,y)t(xy)f^t(x^{-1},x)^{-1}t(x^{-1})))=1$ $\Rightarrow \Gamma^t_{x*y}(h)\Gamma^t_x(\Gamma^t_y(h^{-1})) (\Gamma^t_x(h^{-1})*(f^t(x,y)\\f^t(x^{-1},x)^{-1}f^t(xy,x^{-1})t(xyx^{-1})))=1$
$\Rightarrow \Gamma^t_{x*y}(h) \Gamma^t_x(\Gamma^t_y(h^{-1})) \Gamma^t_{xyx^{-1}}(\Gamma^t_x(h)) = 1~\\\Rightarrow \Gamma^t_{x*y}(h)=\Gamma^t_x(\Gamma^t_y(h))\Gamma^t_y(\Gamma^t_x(h^{-1})) = (\Gamma^t_x *\Gamma^t_y)(h) \Rightarrow \Gamma^t _{x*y} = \Gamma^t_x *\Gamma^t_y \Rightarrow \Gamma^t_{E} ({x*y}) = \Gamma^t_{E}(x) *\Gamma^t_{E} (y)$. Therefore $\Gamma^t_{E}$ is a multiplicative Lie algebra homomorphism.
\end{proof}
Now we will see properties of the function $h^t$ by using equation (1) and (2):

Consider the expression

$ht(x)*(kt(y)\cdot lt(z))=ht(x)*(klf^t(y,z)t(yz))=\Gamma^t_x(klf^t(y,z))\Gamma^t_{yz}(h^{-1})h^t(x,yz))t(x*yz) \hspace{2cm} ...(3)$

On the other hand\\
$(ht(x))*(kt(y)\cdot lt(z))=(ht(x)*kt(y))\cdot ^{kt(y)}(ht(x)*lt(z))=(\Gamma^t_x(k)\Gamma^t_y(h^{-1})h^t(x,y)t(x*y))\cdot$ $(t(y)\Gamma^t_x(l)\Gamma^t_z(h^{-1})h^t(x,z)t(x*z)t(y)^{-1})=(\Gamma^t_x(k)\Gamma^t_y(h^{-1})h^t(x,y)t(x*y))(t(y)$ $\Gamma^t_x(l)\Gamma^t_z(h^{-1})h^t(x,z)t(x*z)f^t(y^{-1},y)^{-1}t(y^{-1}))
=(\Gamma^t_x(k)\Gamma^t_y(h^{-1})h^t(x,y)t(x*y))$ $(\Gamma^t_x(l)\\ \Gamma^t_z(h^{-1})h^t(x,z)f^t(y,x*z)t(y(x*z))f^t(y^{-1},y)^{-1}t(y^{-1}))=(\Gamma^t_x(k)\Gamma^t_y(h^{-1})h^t(x,y)$ $t(x*y))(\Gamma^t_x(l)\Gamma^t_z(h^{-1})h^t(x,z)f^t(y,x*z)f^t(y^{-1},y)^{-1}f^t(y(x*z),y^{-1})t(^y(x*z)))$. 

Finally, we have\\
$ht(x)*(kt(y)\cdot lt(z))=\Gamma^t_x(k)\Gamma^t_y(h^{-1})h^t(x,y)\Gamma^t_x(l)\Gamma^t_z(h^{-1})h^t(x,z)f^t(y,x*z)$ $f^t(y^{-1},y)^{-1}\\f^t(y(x*z),y^{-1})
f^t(x*y,^y(x*z))t((x*y) ^y(x*z)))\hspace{2cm} ...(4)$

Now equating equation (3) and (4), we have 
\begin{equation}\label{3}
\begin{split}
\bold{h^t(x,y)h^t(x,z)f^t(y,x*z)f^t(y^{-1},y)^{-1}f^t(y(x*z),y^{-1})f^t(x*y,^y(x*z))}=\\
\bold{\Gamma^t_x(f^t(y,z))h^t(x,yz)} 
\end{split}
\end{equation}
Similarly, consider\\
$(ht(x)\cdot kt(y))*(lt(z))=(hkf^t(x,y)t(xy))*(lt(z))=\Gamma^t_{xy}(l)\Gamma^t_z(h^{-1}k^{-1}f^t(x,y)^{-1})h^t(xy,z)\\t(xy*z)\hspace{2cm} ...(5)$

On other hand\\
$(ht(x)\cdot kt(y))*(lt(z))=^{ht(x)}(kt(y)*lt(z))\cdot (ht(x)*lt(z))=\Gamma^t_y(l)\Gamma^t_z(k^{-1})h^t(y,z)f^t(x,y*z)f^t(x^{-1},x)^{-1}$ $f^t(x(y*z),x^{-1})t(^x(y*z))\Gamma^t_x(l)\Gamma^t_z(h^{-1})h^t(x,z)t(x*z)=
\Gamma^t_{xy}(l)\Gamma^t_z(h^{-1}k^{-1})\\h^t(y,z) h^t(x,z)$ $f^t(x,y*z)f^t(x^{-1},x)^{-1}f^t(x (y*z),x^{-1})f^t(^x(y*z),x*z)t(^x(y*z)(x*z))\hspace{2cm} ...(6)$

Now equating equation (5) and (6), we have
\begin{equation}\label{4}
\begin{split}
\bold{h^t(x,z)h^t(y,z)f^t(x,y*z)f^t(x^{-1},x)^{-1}f^t(x(y*z),x^{-1})f^t(^x(y*z),x*z)}=\\
\bold{\Gamma^t_z(f^t(x,y)^{-1})h^t(xy,z)}
\end{split}
\end{equation}
Now consider the expressions

$^{lt(z)}(ht(x)*kt(y))=t(z)\Gamma^t_x(k)\Gamma^t_y(h^{-1})h^t(x,y)t(x*y)f^t(z^{-1},z)^{-1}t(z^{-1})$ $=\Gamma^t_x(k)\\ \Gamma^t_y(h^{-1})  h^t(x,y)f^t(z,x*y)t(z(x*y))f^t(z^{-1},z)^{-1}t(z^{-1})$ $= \Gamma^t_x(k)\Gamma^t_y(h^{-1})h^t(x,y)\\ f^t(z,x*y)f^t(z^{-1},z)^{-1}f^t(z(x*y),z^{-1})t(^z(x*y))\hspace{2cm} ...(7)$\\
Also we have \\
$ ^{lt(z)}(ht(x))*^{lt(z)}(kt(y))= (t(z)(ht(x)f^t(z^{-1},z)^{-1}t(z^{-1}))*(t(z)(kt(y))f^t(z^{-1},z)^{-1}\\ t(z^{-1}))$ $=(hf^t(z,x)f^t(z^{-1},z)^{-1}f^t(zx,z^{-1})t(^zx))*(kf^t(z,y)f^t(z^{-1},z)^{-1}f^t(zy,z^{-1})\\ t(^zy))=\Gamma^t_{x}(kf^t(z,y)$ $ f^t(z^{-1},z)^{-1}f^t(zy,z^{-1}))(\Gamma^t_{y}(h^{-1}f^t(z,x)^{-1}f^t(z^{-1},z)f^t(zx,z^{-1})^{-1})\\ h^t(^zx,^zy)t(^zx*^zy)\hspace{2cm} ...(8)$

By equating (7) and (8), we have
\begin{equation}\label{5}
\begin{split}
\bold{\Gamma^t_x(f^t(z,y)f^t(z^{-1},z)^{-1}f^t(zy,z^{-1}))\Gamma^t_y(f^t(z,x)^{-1}f^t(z^{-1},z)f^t(zx,z^{-1})^{-1})h^t(^zx,^zy)} =\\
\bold {h^t(x,y)f^t(z,x*y)f^t(z^{-1},z)^{-1}f^t(z(x*y),z^{-1})}
\end{split}
\end{equation}

Now consider the expression,

$$((ht(x)*kt(y))*^{kt(y)}lt(z))\cdot ((kt(y)*lt(z))*^{lt(z)}ht(x))\cdot ((lt(z)*ht(x))*^{ht(x)}kt(y))=1 $$ $$\Rightarrow (\Gamma^t_{x*y}(lf^t(y,z)f^t(y^{-1},y)^{-1}f^t(yz,y^{-1}))\Gamma^t_z(\Gamma^t_x(k^{-1})\Gamma^t_y(h)h^t(x,y)^{-1})h^t(x*y,^yz)$$ $$t((x*y)*^yz))(\Gamma^t_{y*z}(hf^t(z,x) f^t(z^{-1},z)^{-1}f^t(zx,z^{-1})) \Gamma^t_x(\Gamma^t_y(l^{-1})\Gamma^t_z(k)h^t(y,z)^{-1})$$ $$h^t(y*z,^zx)t((y*z)*^zx))(\Gamma^t_{z*x}(kf^t(x,y)f^t(x^{-1},x)^{-1}f^t(xy,x^{-1}))\Gamma^t_y(\Gamma^t_z(h^{-1})\Gamma^t_x(l)$$ $$h^t(z,x)^{-1})h^t(z*x,^xy)t((z*x)*^xy))=1\Rightarrow \Gamma^t_{x*y}(lf^t(y,z)f^t(y^{-1},y)^{-1}f^t(yz,y^{-1}))$$ $$\Gamma^t_z(\Gamma^t_x(k^{-1})\Gamma^t_y(h)h^t(x,y)^{-1})\Gamma^t_{y*z}(hf^t(z,x)f^t(z^{-1},z)^{-1}f^t(zx,z^{-1})) \Gamma^t_x(\Gamma^t_y(l^{-1})\Gamma^t_z(k)$$ $$h^t(y,z)^{-1})\Gamma^t_{z*x}(kf^t(x,y)f^t(x^{-1},x)^{-1}f^t(xy,x^{-1}))\Gamma^t_y(\Gamma^t_z(h^{-1})\Gamma^t_x(l)   h^t(z,x)^{-1})h^t(x*y,^yz)$$ $h^t(y*z,^zx)f^t((x*y)*^yz,(y*z)*^zx)h^t(z*x,^xy)f^t(((x*y)*^yz)((y*z)*^zx),(z*x)*^xy)=1$

   $$ \Rightarrow \Gamma^t_{x*y}(lf^t(y,z)f^t(y^{-1},y)^{-1}f^t(yz,y^{-1}))\Gamma^t_{y*z}(hf^t(z,x)f^t(z^{-1},z)^{-1}f^t(zx,z^{-1}))$$
   $$\Gamma^t_{z*x}(kf^t(x,y) f^t(x^{-1},x)^{-1}f^t(xy,x^{-1}))
   \Gamma^t_{x*z}(k)\Gamma^t_{z*y}(h)\Gamma^t_{y*x}(l)\Gamma^t_{z}(h^t(x,y)^{-1})\Gamma^t_x(h^t(y,z)^{-1}))$$ $\Gamma^t_y(h^t(z,x)^{-1}))h^t(x*y,^yz)h^t(y*z,^zx) f^t((x*y)*^yz,(y*z)*^zx)h^t(z*x,^xy)f^t(((x*y)*^yz)\\((y*z)*^zx),(z*x)*^xy)=1$.
    
Thus we have
\begin{equation}\label{6}
\begin{split}
\bold{\Gamma^t_{x*y}(f^t(y,z)f^t(y^{-1},y)^{-1}f^t(yz,y^{-1}))\Gamma^t_{y*z}(f^t(z,x)f^t(z^{-1},z)^{-1}f^t(zx,z^{-1}))}\\
 \bold{\Gamma^t_{z*x}(f^t(x,y)f^t(x^{-1},x)^{-1}f^t(xy,x^{-1}))\Gamma^t_{z}(h^t(x,y)^{-1})
 \Gamma^t_{x}(h^t(y,z)^{-1})\Gamma^t_y(h^t(z,x)^{-1})}\\\bold{h^t((x*y),^yz)h^t((y*z),^zx)h^t((z*x),^xy)                                                        f^t((x*y)*^yz,(y*z)*^zx)f^t(((x*y)*^yz)}\\ \bold{((y*z)*^zx),(z*x)*^xy)=1} \
 \end{split}
 \end{equation}
\begin{definition} \label{center factor system}
Let $H$ be an abelian group with trivial multiplicative Lie algebra structure and $K$ be an arbitrary multiplicative Lie algebra.  Then a \textbf{center factor system} is a quintuple $(K,H,f,h,\Gamma)$, where $\Gamma : K \longrightarrow End(H)$ is a multiplicative Lie algebra homomorphism and $f, \ h$ are maps from $K \times K$ to $H$ satisfying the conditions like equations $(\ref{1}), (\ref{2}), (\ref{3}), (\ref{4}), (\ref{5})$ and $(\ref{6})$.
\end{definition}
\begin{proposition} \label{center}
Every center extension $E(H, K)$ with a choice of section $t$ determines a center factor system $\text{CFac}(E,t)=(K, H, f^t,h^t,\Gamma^t_{E})$ and conversely for a given center factor system $(K, H, f,h,\Gamma)$, their exists a center extension $E(H, K)$ of $H$ by $K$ with a section $t$ such that $(K, H, f,h,\Gamma)=\text{CFac}(E,t)$.
\end{proposition} 
\begin{proof}
From the above discussions and the Definition \ref{center factor system}, it can be seen that every center extension $E(H, K)$ of $H$ by $K$ with a given section $t$ determines a center factor system $CFac(E, t) \equiv (K,H,\Gamma^t,f^t,h^t)$ (this is called as a center factor system given by the center extension E).

Conversely, let $G=H \times K.$ Define $\cdot$ and $*$ operations on $G$ as
\begin{center}
$(a,x)\cdot (b,y)=(abf(x,y),xy)$
and $(a,x)*(b,y)=(\Gamma_x(b)\Gamma_y(a^{-1})h(x,y),x*y)$
\end{center}
It is easy to see that $(G,\cdot,*)$ is a multiplicative Lie Algebra such that
\begin{center}
\begin{tikzcd}
   E(H, K) \equiv 1\arrow{r} & H\arrow{r}{i} & G \arrow{r}{p} & K\arrow{r} & 1
\end{tikzcd}
\end{center}
is center extension of $H$ by $K$, where $i$ is the first inclusion and $p$ is the second projection. Let $t$ be a section of $E$ given by,
\begin{center}
$t(x)=(1,x)$
\end{center}
By an easy computation, it can be seen that $\Gamma^t_{E}=\Gamma$, $f^t=f$ and $h^t=h$. 
\end{proof}
\subsection*{Equivalence between Category CEXT of center extensions and Category CFAC of center factor systems}
 Let  $(\lambda, \mu, \nu)$ be the morphism between the two center extensions $E(H_1, K_1)$ and $E(H_2, K_2)$, given by the following commutative diagram: 
\begin{center}
\begin{tikzcd}
   E(H_1, K_1) \equiv 1\arrow{r} & H_1 \arrow{r}{i_1}\arrow{d}{\lambda} & G_1 \arrow{r}{\beta_1}\arrow{d}{\mu} & K_1\arrow{r}\arrow{d}{\nu} & 1 \\
   E(H_2, K_2) \equiv 1\arrow{r} & H_2\arrow{r}{i_2} & G_2\arrow{r}{\beta_2} & K_2\arrow{r} & 1
\end{tikzcd}
\end{center}
Let $t_1$ and $t_2$ be sections of $E(H_1, K_1)$ and $E(H_2, K_2)$ respectively. Consider the corresponding factor systems $(K_1,H_1,f^{t_1},h^{t_1}, \Gamma^{t_1})$ and $(K_2,H_2,f^{t_2},h^{t_2}, \Gamma^{t_2})$ of $E(H_1, K_1)$ and $E(H_2, K_2)$ respectively. Then for $x \in K_1$, $\mu(t_1(x))\in G_2$ and $\beta_2(\mu(t_1(x)))$ = $\nu(\beta_1(t_1(x)))$ = $\nu(x)$ = $\beta_2(t_2(\nu(x)))$. Thus $\mu(t_1(x))(t_2(\nu(x)))^{-1}$ $ \in ker \beta_2(= H_2)$. In turn, we have a unique $g(x)$ $\in H_2$ such that $g(x)$ = $ \mu(t_1(x))(t_2(\nu(x)))^{-1}$. This implies 
\begin{center}
$\mu(t_1(x))$ = $ g(x)t_2(\nu(x))\hspace{2cm} ... (9)$
\end{center}
Since $t_1(1)=1$ = $t_2(1)$,  it follows that
\begin{equation}\label{7}
\bold{g(1)=1}
\end{equation}
Since $\mu(t_1(x)t_1(y))=\mu(t_1(x))\mu(t_1(y))$, we have the following equation
\begin{equation}\label{8}
\bold{\lambda(f^{t_1}(x,y))g(xy) = g(x)g(y)f^{t_2}(\nu(x),\nu(y))}
\end{equation}
Also we have 
$\mu(t_1(x)*t_1(y))=\mu(h^{t_1}(x,y)t_1(x*y))=\lambda(h^{t_1}(x,y))\mu(t_1(x*y))$
$=\lambda(h^{t_1}(x,y))\mu(t_1(x*y))=\lambda(h^{t_1}(x,y))g(x*y)t_2(\nu(x)*\nu(y))\hspace{2cm} ... (10)$

On the other hand $\mu(t_1(x)*t_1(y))=\mu(t_1(x))*\mu(t_1(y))=(g(x)t_2(\nu(x)))*(g(y)t_2(\nu(y)))$
$=\Gamma_{\nu(x)}^{t_2}(g(y))\Gamma_{\nu(y)}^{t_2}(g(x)^{-1})h^{t_2}(\nu(x),\nu(y))t_2(\nu(x)*\nu(y))\hspace{2cm} ... (11)$ 

On comparing equations (10) and (11), we have
\begin{equation}\label{9}
\bold{\lambda(h^{t_1}(x,y))g(x*y)=\Gamma_{\nu(x)}^{t_2}(g(y))}\bold {\Gamma_{\nu(y)}^{t_2}(g(x)^{-1})h^{t_2}(\nu(x),\nu(y))}
\end{equation}
Thus a morphism $(\lambda,\mu,\nu)$ between two center extensions $E(H_1, K_1)$ and $E(H_2, K_2)$ together with choices of sections $t_1$ and $t_2$ of the corresponding extensions, induces a map $g$ from $K_1$ to $H_2$  such that the triple $(\nu, g, \lambda)$ satisfies equations (\ref{7}), (\ref{8}) and (\ref{9}). It can be seen as a morphism from the factor system $(K_1,H_1,f^{t_1},h^{t_1}, \Gamma^{t_1}_{E_1})$ to $(K_2,H_2,f^{t_2},h^{t_2}, \Gamma^{t_2}_{E_2})$.

Let $(\lambda_1,\mu_1,\nu_1)$ be a morphism from an extension
\begin{center}
\begin{tikzcd}
E(H_1, K_1) \equiv   1\arrow{r} & H_1\arrow{r}{i_1} & G_1 \arrow{r}{\beta_1} & K_1\arrow{r} & 1
\end{tikzcd}
\end{center}
to an extension
\begin{center}
\begin{tikzcd}
E(H_2, K_2) \equiv   1\arrow{r} & H_2\arrow{r}{i_2} & G_2 \arrow{r}{\beta_2} & K_2\arrow{r} & 1
\end{tikzcd}
\end{center}
and $(\lambda_2,\mu_2,\nu_2)$ be another morphism from $E_2$ to
\begin{center}
\begin{tikzcd}
E(H_3, K_3) \equiv   1\arrow{r} & H_3\arrow{r}{i_3} & G_3 \arrow{r}{\beta_3} & K_3\arrow{r} & 1
\end{tikzcd}
\end{center}
Let $t_1, ~ t_2,~\text{and} ~ t_3$ be corresponding choices of sections. Then
$\mu_1(t_1(x))=g_1(x)t_2(\nu_1(x))$ and $\mu_2(t_2(x))=g_2(x)t_3(\nu_2(x))$, where $g_1:K_1 \longrightarrow H_2$ and $g_2:K_2 \longrightarrow H_3$ are uniquely determined maps same as $g$ in equation (9).
In turn, we have, 

$\mu_2(\mu_1(t_1(x)))=\mu_2((g_1(x)))\mu_2(t_2(\nu_1(x)))=\mu_2(g_1(x)))g_2(\nu_1(x))t_3(\nu_2(\nu_1(x)))$ = $\lambda_2 (g_1(x))g_2(\nu_1(x))t_3(\nu_2(\nu_1(x)))=g_3(x)t_3(\nu_2(\nu_1(x)))$, where $g_3(x) = \lambda_2 (g_1(x))g_2(\nu_1(x))$, for each $x \in K_1$. Thus the composition $(\lambda_2 \circ \lambda_1, \mu_2 \circ \mu_1,\nu_2 \circ \nu_1)$ induces the triple $(\nu_2 \circ \nu_1,g_3, \lambda_2 \circ \lambda_1)$. 

Now we introduce the category CFACS  whose objects are center factor system, and a morphism from $(K_1,H_1,f^{1},h^{1}, \Gamma^1)$ to $(K_2,H_2,f^{2},h^{2}, \Gamma^2)$ is a triple $(\nu, g, \lambda)$, where $\nu : K_1 \longrightarrow K_2$, $\lambda: H_1 \longrightarrow H_2$ are multiplicative Lie algebra homomorphism, and $g: K_1 \longrightarrow H_2$ is a map such that 
\begin{enumerate}
\item $g(1) = 1$
\item $\lambda(f^{1}(x,y))g(xy)$ = $g(x)g(y)f^{2}(\nu(x),\nu(y))$
\item $\lambda(h^{1}(x,y))g(x*y)=\Gamma_{\nu(x)}^{2}(g(y))\Gamma_{\nu(y)}^{2}(g(x)^{-1})h^{2}(\nu(x),\nu(y))$
\end{enumerate}
The composition of morphisms $(\nu_1, g_1, \lambda_1)$ from $(K_1,H_1,f^{1},h^{1}, \Gamma^{1})$ to $(K_2,H_2,f^{2},\\h^{2}, \Gamma^{2})$ and $(\nu_2, g_2, \lambda_2)$ from $(K_2,H_2,f^{2},h^{2}, \Gamma^2)$ to $(K_3,H_3,f^{3},h^{3}, \Gamma^3)$ is the triple $(\nu_2 ~ o~ \nu_1, g_3, \lambda_2 ~ o~ \lambda_1)$, where $g_3$ is given by $g_3 (x) = g_2(\nu_1 (x))\lambda_2(g_1(x))$
for each $x \in K_1$.\\
So, finally from above discussion we have the following theorem:
\begin{theorem} \label{equivalence}
There is an equivalence between the category $\mathbf{CEXT}$ of center extensions to the category $\mathbf{CFAC}$ of center factor systems.
\end{theorem}

%Now we can give the definition of equivalence between two factor systems motivated from that of equivalent extensions:
%\begin{definition}
 %Two center factor systems $(K,H,f^{t_1},h^{t_1}, \Gamma^1)$ and $(K,H,f^{t_2},h^{t_2}, \Gamma^2)$ are said to be equivalent if there exists a map $g$ from $K$ to $H$ satisfying,
%\begin{enumerate}
%\item $g(1)=1$
%\item $f^{t_1}(x,y)g(xy)$ = $g(x)g(y)f^{t_2}(x,y)$
%\item $h^{t_1}(x,y)g(x*y)=\Gamma_{x}^2(g(y))\Gamma_{y}^2(g(x)^{-1})h^{t_2}(x,y)$
%\end{enumerate}
%\end{definition}
Now we extend the Lemma 2.2 to illustrate the equivalence classes of center extensions.
\begin{lemma} \label{independence of section and representative}
 Let $E(H, K)$ be an center extension and $t$ be a section of $E(H, K)$. Then the map $\Gamma^t_E$ is independent of choice of section as well as choice of the representative $E$ of the equivalence class.
\end{lemma}
\begin{proof}
Let s be an another section of $E(H, K)$ and $\Gamma^{s}_{E} : K \longrightarrow End(H)$ be the corresponding multiplicative Lie algebra homomorphism . Since $s$ and $t$ are two sections of $E(H, K)$, their exists a map $g: K\rightarrow H$ with $g(1)=1$ such that $s(x)=g(x)t(x)$ for all $ x \in K$. Therefore for $x \in K$, $\Gamma^s_x(h) = s(x)*h = (g(x)t(x))*h = ^{g(x)}(t(x)*h) (g(x)*h) = t(x)*h = \Gamma^t_x (h)$. Thus $\Gamma^{s}_{E} (x) = \Gamma^{t}_{E} (x)$ for all $x \in K$. This shows that $\Gamma^t_E$ is independent on section $t$. 

Now let $E'(H, K)$ be an equivalent center extension to $E(H, K)$. Let $t'$ be a section of $E'(H, K)$, $\Gamma^{t'}_{E'} : K \longrightarrow End(H)$ be the corresponding multiplicative Lie algebra homomorphism and $(I_H, \phi, I_K)$ be an equivalence from $E(H, K)$ to $E'(H, K)$.

Since $t$ is section of $E(H, K)$,  $\phi \circ t$ is a section of $E'(H, K)$. Therefore their exists a map $g$ from $K$ to $H$ satisfying $g(x)\phi(t(x)) = t'(x)$. Now we have
$\Gamma^{t'}_x(h) = t'(x)*h = (g(x)\phi(t(x)))*h = ^{g(x)}(\phi(t(x))*h) (g(x)*h) = \phi(t(x))*h = \phi(t(x))*\phi(h) = \phi(t(x)*h) = \phi(\Gamma^t_x (h)) = \Gamma^t_x (h)$ (since $\phi (h) = h$ for all $h \in H$). Thus $\Gamma^{t'}_{E'} (x) = \Gamma^{t}_{E} (x)$ for all $x \in K$.
Therefore $\Gamma^t_{E}$ is independent of choice of sections and the equivalent extensions.

Now onwards, for a center extension $E(H, K)$, we denote without any ambiguity, $\Gamma^t_E$ by $\Gamma_{[E]}$. 
\end{proof}
\begin{theorem} \label{surjective map}
Let $CExt(H, K)$ denote the set of all equivalence classes of center extensions of $H$ by $K$ and $Hom(K, End(H))$ denote the set of all multiplicative Lie algebra homomorphism from $K$ to $End(H)$. Then there is a natural surjective map $\eta : CExt(H, K) \longrightarrow Hom(K, End(H))$ defined by $\eta([E]) = \Gamma_{[E]}$.
\end{theorem}
\begin{proof}
From the Lemma \ref{independence of section and representative}, the map $\eta :CExt(H, K) \rightarrow Hom(K, End(H))$ defined by $\eta([E]) = \Gamma_{[E]}$ is well defined. 
 
Let $\Gamma : K\longrightarrow End(H)$ is a multiplicative Lie algebra homomorphism defined by $\Gamma(x) = \Gamma_x$. Let $G = H \times K$. Then $G$ is a multiplicative Lie algebra with the operations $\cdot $ and $*$, defined by $(h, x)\cdot (k, y) = (hk,xy)$ and $(h, x)*(k, y) = (\Gamma_x(k)\Gamma_y(h^{-1}), x*y)$ respectively. Consider a section $t : K \longrightarrow G$ is defined by $t(x) = (1, x)$. Then $t$ is a group homomorphism. Therefore we have a center extension $E$ of multiplicative Lie algebra of $H$ by $K$
\begin{center}
\begin{tikzcd}
   1\arrow{r} & H\arrow{r}{\alpha} & G \arrow{r}{\beta} & K\arrow{r} & 1
\end{tikzcd}
\end{center}
such that $f^t$ and $h^t$ are trivial maps and $\Gamma_E = \Gamma$.
Thus every multiplicative Lie algebra homomorphism from $K\longrightarrow End(H)$ gives a  center extension of multiplicative Lie algebra of $H$ by $K$. This shows that $\eta$ is surjective.
\end{proof}
\begin{example}
The map $\eta$ in Theorem $\ref{surjective map}$ need not be injective. Specifically, there may exists same multiplicative Lie algebra homomorphism  $\Gamma$ from $K$ to $End(H)$ for two distinct classes of extensions. Consider the center extensions:
\begin{center}
\begin{tikzcd}
E_1(\mathbb{Z},\mathbb{Z}_5) \equiv   \lbrace 0\rbrace \arrow{r} & \mathbb{Z}\arrow{r}{(i_1,o)} & \mathbb{Z}\bigoplus \mathbb{Z}_5 \arrow{r}{p_2} & \mathbb{Z}_5\arrow{r} &\lbrace 0\rbrace
\end{tikzcd}
\end{center}
and
\begin{center}
\begin{tikzcd}
E_2(\mathbb{Z},\mathbb{Z}_5) \equiv   \lbrace 0\rbrace\arrow{r} & \mathbb{Z}\arrow{r}{\alpha} & \mathbb{Z} \arrow{r}{\nu} & \mathbb{Z}_5\arrow{r} & \lbrace 0\rbrace
\end{tikzcd}
\end{center}
of the multiplicative Lie algebra $\mathbb{Z}$ by $\mathbb{Z}_5$, where $\alpha(n)=5n, ~ \forall ~ n \in \mathbb{Z}$ and $\nu$ is natural quotient map. Since $\mathbb{Z}$ and $\mathbb{Z}_5$ are cyclic groups, the multiplicative Lie algebra structure on these groups is trivial. Also it can be seen that $E_1(\mathbb{Z},\mathbb{Z}_5)$ and $E_2(\mathbb{Z},\mathbb{Z}_5)$ are not equivalent as center extension. Since $End(\mathbb{Z})$ is isomorphic to $\mathbb{Z}$, $Hom(\mathbb{Z}_5, \mathbb{Z})$ is trivial. Therefore here it can be seen that $[E_1]\neq[E_2]$ but $\eta([E_1])=\eta([E_2])$.
\end{example}
Now we discuss the following problem:

\noindent\textbf{Problem:} Let $H$ be an abelian group with trivial multiplicative Lie algebra structure. Classify all center extensions of $H$ by $K$ (up to equivalence) with the given 
multiplicative Lie algebra homomorphism $\Gamma : K \longrightarrow End(H)$
\begin{definition}  
A multiplicative center 2-cocycle of a multiplicative Lie algebra $K$ with coefficient in an abelian group $H$ with trivial multiplicative Lie algebra structure is a triplet $(f,h,\Gamma)$, where $f \in Z^2(K,H)$ is a group 2-cocycle of $K$ with coefficient in the trivial $K$-module $H$, $ \Gamma:K\longrightarrow End(H)$ is a multiplicative Lie algebra homomorphism and $h:K \times K \longrightarrow H$ is a map which satisfying the conditions like equations $(\ref{1})$, $(\ref{2})$, $(\ref{3})$, $(\ref{4})$, $(\ref{5})$ and $(\ref{6})$, where we replace $f^t$ and $h^t$ with $f, ~ h$ respectively.
\end{definition}
\begin{remark} \label{multiplicative center cycle}
Let $E(H, K)$ be a center extension with a choice of section $t$. Then from above discussion, it is clear that we have a multiplicative center 2-cocycle $(f^t,h^t,\Gamma_E)$.

Conversely, if we have a multiplicative center 2-cocycle  $(f,h, \Gamma)$ of multiplication Lie Algebra $K$ with coefficient in the abelian group $H$ with trivial multiplicative Lie algebra structure, then from Proposition $(\ref{center})$ it can be seen that there exists a multiplicative Lie algebra $G$ and a center extension $E(H, K)$ with a choice of section $t$ such that $f^t=f$, $\Gamma_{E}=\Gamma$ and $h^t=h$.
\end{remark}
Let $E(H, K)$ be a center extension with a choice of section $t$. Then we have a multiplicative center 2-cocycle $(f^t,h^t,\Gamma_E)$.
Now, let s be an another section of $E$. Then their exists a map $g: K\rightarrow H$ with $g(1)=1$ such that $s(x)=g(x)t(x)$ for all $ x \in K$. So 
$$(hs(x))\cdot (ks(y))=hkf^s(x,y)s(xy)=hkf^s(x,y)g(xy)t(xy) \cdots (12)$$
On the other hand
$$(hs(x))\cdot (ks(y))=(hg(x)t(x))\cdot (kg(y)t(y))=hkg(x)g(y)f^t(x,y)t(xy) \cdots (13)$$
On comparing equations $(12)$ and $(13)$
$$ \bold{f^s(x,y)=g(x)g(xy)^{-1}g(y)f^t(x,y)}$$

Consider the expression

$(hs(x))*(ks(y))=\Gamma_x(k)\Gamma_y(h^{-1})h^s(x,y)s(x*y)=\Gamma_x(k)\Gamma_y(h^{-1})h^s(x,y)g(x*y)t(x*y)... (14)$

On the other hand 

$(hs(x))*(ks(y))=(hg(x)t(x))*(kg(y)t(x))
= \Gamma_x(kg(y))\Gamma_y(h^{-1}g(x)^{-1})h^t(x,y)t(x*y)... (15)$

On comparing equations $(14)$ and $(15)$ we have
$$\bold{h^s(x,y)=\Gamma_x(g(y))\Gamma_y(g(x)^{-1})g(x*y)^{-1}h^t(x,y)}$$

 This motivates us for the following definition:
\begin{definition}
Two multiplicative center 2-cocycle $(f^s,h^s,\Gamma)$ and $(f^t,h^t,\Gamma)$ are said to be equivalent if their exists an identity preserving map $g:K\longrightarrow H$ satisfying:
\begin{enumerate}
\item $f^s(x,y)=g(x)g(y)g(xy)^{-1}f^t(x,y)$
\item $h^s(x,y)=\Gamma_x(g(y))\Gamma_y(g(x)^{-1})g(x*y)^{-1}h^t(x,y)$
\end{enumerate}
\end{definition}
The set $Z^2_{ML(\Gamma)}(K,H)$ of all multiplicative center 2-cocycles of $K$ with coefficient in $H$ is easily seen to be an abelian group with respect to coordinate wise operation given by $(f,h, \Gamma) \cdot (f^{'},h^{'}, \Gamma)=(ff^{'},hh^{'}, \Gamma)$. Given any identity preserving map $g:K \longrightarrow H$, the triplet $(\delta g,g^{*}, \Gamma)$  is a member of $Z^2_{ML(\Gamma)}(K,H)$, where $\delta g,g^{*}$ are maps from $K \times K$ to $H$ given by
$\delta g((x,y))=g(y)g(xy)^{-1}g(x)$ and $g^*(x,y)=\Gamma_x(g(y))\Gamma_y(g(x)^{-1})g(x*y)^{-1}$.

Let $MAP(K,H)$ denote the group of identity preserving map from $K$ to $H$. So we have a homomorphism $\chi:MAP(K,H)\longrightarrow Z^2_{ML(\Gamma)}(K,H)$ given by $\chi(g)=(\delta g,g^*, \Gamma)$.

The image of $\chi$ is called the group of multiplicative center 2-coboundaries of $K$ with coefficient in $H$ and it is denoted by $B^2_{ML(\Gamma)}(K,H)$. The quotient group $\frac{Z^2_{ML(\Gamma)}(K,H)}{ B^2_{ML(\Gamma)}(K,H)}$ is called the second center cohomology of $K$ with coefficient in $H$ and it is denoted by $H^2_{ML(\Gamma)}(K,H)$. In turn, we get the following exact sequence of abelian groups,
$ 1\rightarrow  Hom(K,H)\xrightarrow{i} MAP(K,H)\xrightarrow{\chi} Z^2_{ML(\Gamma)}(K,H)\xrightarrow{\nu} H^2_{ML(\Gamma)}(K,H)\rightarrow 1$, 
where $\nu$ is quotient map.
\begin{theorem} \label{Bijective correspondence}
Let $H$ be an abelian group with trivial multiplicative Lie algebra structure, $K$ be a multiplicative Lie algebra and $\Gamma$ be a multiplicative Lie algebra homomorphism from $K$ to $End(H)$. Then there is a natural bijective correspondence between the set $CExt_\Gamma (H,K)$ of equivalence classes of center extensions of $H$ by $K$ with given $\Gamma$ and the second Lie co- homology $H^2_{ML(\Gamma)}(K,H)$.  
\end{theorem}
\begin{proof}
Let $E(H, K)~ \text{and} ~ E'(H, K) $ be two equivalent center extensions of $H$ by $K$ with multiplicative Lie algebra homomorphism $\Gamma$ and sections $t, ~ t'$ respectively. Let $(f^{t},h^{t}, \Gamma)$ and $(f^{t'},h^{t'}, \Gamma)$ be the multiplicative center 2-cocycles corresponding to center extensions $E(H, K)$ and $E'(H, K)$ respectively. Then their exists a map $g$ from $K$ to $H$ with $g(1)=1$ such that $f^{t'}(x,y)=g(x)g(xy)^{-1}g(y)f^t(x,y)$ and $h^{t'}(x,y)=\Gamma_x(g(y))\Gamma_y(g(x)^{-1})g(x*y)h^t(x,y)$.

This shows that $(f^t,h^t,\Gamma)B^2_{ML(\Gamma)}(K,H)=(f^{t'},h^{t'},\Gamma)B^2_{ML(\Gamma)}(K,H)$. Therefore we have a map $\Phi$ from $CExt_\Gamma (H,K)$ to $H^2_{ML(\Gamma)}(K,H)$ defined by $\Phi ([E])=(f^t,h^t,\Gamma) \\ B^2_{ML(\Gamma)}(K,H)$, where $t$ is a section of $E(H, K)$. Let $(f,h,\Gamma)~ \in Z^2_{ML(\Gamma)}(K,H)$. Then by the Remark $(\ref{multiplicative center cycle})$, we have a center extension $E(H, K)$ of $H$ by $K$ and section $t$ such that $f^t=f$, $h^t=h$ and $\Gamma_E=\Gamma$. This shows that $\Phi$ is surjective. Let $E(H, K)$ and $E'(H, K)$ be two center extensions of $H$ by $K$, with sections $t$ and $t'$ respectively such that $\Phi([E])=\Phi([E'])$, then   $(f^t,h^t,\Gamma)B^2_{ML(\Gamma)}(K,H)=(f^{t'},h^{t'},\Gamma)B^2_{ML(\Gamma)}(K,H)$. Hence again their exists $g$ from $K$ to $H$ with $g(1)=1$ satisfying $f^{t'}(x,y)=g(x)g(xy)^{-1}g(y)f^t(x,y)$ and $h^{t'}(x,y)=\Gamma_x(g(y))\Gamma_y(g(x)^{-1})g(x*y)h^t(x,y)$. It follows that $(f^{t},h^{t}, \Gamma)$ and $(f^{t'},h^{t'}, \Gamma)$ are equivalent. Therefore the center factor systems $(K, H, f^{t},h^{t}, \Gamma)$ and $(K, H, f^{t'},h^{t'}, \Gamma)$ are equivalent and hence  $E(H, K)$ and $E'(H, K)$ are equivalent. Thus $\Phi$ is injective.
\end{proof}
\begin{remark}
Let $(G_1,\circ_1,*_1)$ and $(G_2,\circ_2,*_2)$ be two multiplicative Lie algebras. Then $G = G_1 \times G_2$ is also a multiplicative Lie algebra with operations $\cdot$ and $*$ given by 
\begin{center}
$(g_1,g_2)\cdot (g'_1,g'_2)=(g_1\circ_1 g'_1,g_2\circ_2 g'_2)$ and
$(g_1,g_2)* (g'_1,g'_2)=(g_1 *_1 g'_1,g_2 *_2 g'_2)$
\end{center}
 \end{remark}
\noindent\textbf{Baer sum on the class of center extensions-:} Now we will define Baer sum on class of center extensions of multiplicative Lie algebras analogues to class of group extensions. Let  
\begin{center}
\begin{tikzcd}
  E_1(H, K) \equiv 1\arrow{r} & H\arrow{r}{i_1} & G \arrow{r}{\beta_1} & K\arrow{r} & 1
\end{tikzcd} 
\end{center}
and 
\begin{center}
\begin{tikzcd}
  E_2(H, K) \equiv 1\arrow{r} & H\arrow{r}{i_2} & G' \arrow{r}{\beta_2} & K\arrow{r} & 1
\end{tikzcd} 
\end{center}
be two center extensions of $H$ by $K$ with the corresponding fixed multiplicative Lie algebra homomorphism $\Gamma : K \longrightarrow End(H)$. Consider $L=\lbrace (g_1,g_2):g_1 \in G, g_2 \in G',\beta_1(g_1)=\beta_2(g_2)\rbrace$ and $D=\lbrace (h,h^{-1}):h \in H \rbrace$. Then it is easy to verify that $L$ is a subalgebra of $G_1 \times G_2$ and $D$ is an ideal of $L$. Also we get a center extension $E_1(H,K)\biguplus E_2(H,K)$ of $H$ by $K$.
\begin{center}
\begin{tikzcd}
  E_1(H, K)\biguplus E_2(H,K) \equiv 1\arrow{r} & H\arrow{r}{\eta} & \bar{G} \arrow{r}{\bar{\chi}} & K\arrow{r} & 1
\end{tikzcd} 
\end{center}
where $\bar{G}=\frac{L}{D}$, $\eta (h)=(h,1)D$ and $\bar{\chi}((g_1,g_2)D)=\beta_1(g_1)$ are multiplicative Lie algebra homomorphisms (for the group structure details, see pages 388 \& 389 of \cite{RL}). Further, let $t_1$ be a section of $E_1(H,K)$ and $t_2$ be section of $E_2(H,K)$ with corresponding multiplicative center 2-cocycles $(f^{t_1}, h^{t_1},\Gamma)$ and $(f^{t_2},h^{t_2},\Gamma)$ respectively. Then we have a section $t_1+t_2$ of $E_1(H, K)\biguplus E_2(H,K)$ given by $(t_1+t_2)(x)=(t_1(x),t_2(x))D$. 

\textbf{Claim:} $f^{t_1+t_2}=f^{t_1}f^{t_2}$, $h^{t_1+t_2}=h^{t_1}h^{t_2}$ and $\Gamma_{E_1\biguplus E_2}=\Gamma$.

Let $x,y \in K$. Then 
$\eta(h^{t_1+t_2}(x,y))=((t_1+t_2)(x)*(t_1+t_2)(y))((t_1+t_2)(x*y))^{-1}\Rightarrow (h^{t_1+t_2}(x,y), 1)D = ((t_1(x),t_2(x))*(t_1(y),t_2(y)))((t_1(x*y))^{-1},(t_2(x*y))^{-1})D=((t_1(x)*t_1(y))(t_1(x*y))^{-1},(t_2(x)*t_2(y))(t_2(x*y))^{-1})D=(h^{t_1}(x,y),h^{t_2}(x,y))D \Rightarrow (h^{t_1+t_2}(x,y),1)D=(h^{t_1}(x,y),h^{t_2}(x,y))D \Longrightarrow (h^{t_1+t_2}(x,y)(h^{t_1}(x,y))^{-1},(h^{t_2}(x,y))^{-1}) \\\in D$. Therefore $h^{t_1+t_2}=h^{t_1}h^{t_2}.$

Similarly, we can show that $f^{t_1+t_2}=f^{t_1}f^{t_2}$.

Also $\eta(\Gamma_x^{t_1+t_2}(h))=(t_1+t_2)(x)*\eta(h)=((t_1(x),t_2(x))*(h,1))D  =(t_1(x)*h,1)D=(\Gamma_x(h),1)D \Rightarrow (\Gamma^{t_1+t_2}(h), 1)D = (\Gamma_x(h),1)D$. Hence $\Gamma_{E_1\biguplus E_2}=\Gamma$.
 
 Therefore  $CExt_\Gamma(H,K)$ is an abelian group with respect to the Baer sum and the bijective map $\Phi$ defined in Theorem $(\ref{Bijective correspondence})$ from the set $CExt_\Gamma(H,K)$ to $H^2_{ML(\Gamma)}(H,K)$ is an isomorphism. 
 
 \begin{example}
Consider the dihedral group $D_4 = \{<a, b> \mid a^2 = 1 = b^4, aba = b^{-1}\}$, let $H= \lbrace 1, b^2\rbrace $ and $K = \frac{D_4}{H}$. Define the multiplicative Lie algebra structure $*$ on $D_4$ by, $a*a=b*b=1$, and $a*b=b$. Then $(D_4,\cdot ,*)$ forms a multiplicative Lie algebra and the induced structure on $K$ is given by $aH*aH=bH*bH=H$, and $aH*bH=bH$.
So, finally we have a center extension 
\begin{center}
\begin{tikzcd}
   1\arrow{r} & H\arrow{r}{i} & D_4 \arrow{r}{\nu} & K\arrow{r} & 1
\end{tikzcd}
\end{center}
of $H$ by $K$. Let $t:K\longrightarrow D_4$ be a section of the above extension defined by $t(H)=e$, $t(aH)=a$, $t(bH)=b^3$, and $t(abH)=ab$.  \\
Then the multiplicative center 2-cocycle $(f^t, h^t, \Gamma)$ are as follows:\\
$f^t(x, y)=  \begin{cases} 
      1 & (x,y) \notin \{(aH,abH), (aH,bH),(bH,abH)\}\\
      b^2 & (x,y)\in \{(aH,bH),(bH,abH),(aH,abH)\} 
   \end{cases}$, 
  \\
   $h^t(x, y)=  \begin{cases} 
      1 & (x,y) \notin \{(bH,aH),(abH,aH),(bH,abH)\}\\
      b^2 & (x, y)\in \{(bH,aH),(abH,aH),(bH,abH)\} 
   \end{cases}$
   \\
$\Gamma(x)=  \begin{cases} 
      0_H & x \in \{H, bH\}\\
      I_H & x \in \{aH, abH\} 
   \end{cases}$ \\
where $0_H$ and $I_H$ denotes the zero homomorphism and identity   homomorphism on $H$ respectively.
\end{example}
\section{Schreier Theory for Lie Center Extension} 
In this section, we discuss the theory of Lie center extension. 
\begin{definition}
An extension \begin{tikzcd}
  E(H, K) \equiv 1\arrow{r} & H\arrow{r}{i} & G \arrow{r}{\beta} & K\arrow{r} & 1
\end{tikzcd} of an abelian group $H$ with trivial multiplicative Lie algebra structure by a multiplicative Lie algebra $K$ is called a \textbf{Lie center extension} if $H$ is contained in the Lie center $LZ(G) (=\lbrace x \in G : x*y=1, \forall y \in G \rbrace )$ of $G$.
\end{definition}
Let 
\begin{tikzcd}
  E(H, K) \equiv 1\arrow{r} & H\arrow{r}{i} & G \arrow{r}{\beta} & K\arrow{r} & 1
\end{tikzcd} be a Lie center extension of $H$ by $K$ and $t : K \longrightarrow G$ be a section of $E(H, K)$. Then every element of group $G$ can be uniquely expressed in the form $ht(x)$, for some $h \in H$ and $x \in K$. 

The group operation $``\cdot" $ (for details see \cite{RL}) in $G$ is given by $$\mathbf{(ht(x))\cdot (kt(y))=h \sigma_x(k)f^t(x,y)t(xy)} ... (16)$$ where $\sigma_x^t : H \longrightarrow H$ is an automorphism defined by $\sigma_x^t(k) = t(x)kt(x)^{-1}$ and $f^t : K \times K \rightarrow H$ is a map satisfying 
\begin{equation}\label{10}
\mathbf{f^t(1,x)=f^t(x,1)=1 ~\text{and}~ f^t(x,y)f^t(xy,z) = \sigma_x^t(f^t(y,z))f^t(x,yz)}
\end{equation}
\begin{remark}  Since $H$ is abelian, $Aut(H) = Out(H)$. 
The map $\sigma_E^t : K \longrightarrow Aut(H)$ defined by $\sigma_E^t (x) = \sigma_x^t$ is independent of the section $t$ and the representative of extension $E(H, K)$ of the equivalence class (see (\cite{RL}, Page 378)). So further in this paper we will use $\sigma_x$ in place of $\sigma_x^t$. 
\end{remark}
 
Since $\beta$ is a multiplicative Lie algebra homomorphism, for $x,y \in K$, $\beta(t(x)*t(y))=\beta(t(x))*\beta(t(y))=x*y=\beta t(x*y)$
$\Rightarrow (t(x)*t(y))t(x*y)^{-1} \in \text{ker} \beta ~(= H) $
 
So we have a map $h^t~ :~ K \times K \rightarrow H$ such that
$ t(x)*t(y)=h^t(x,y)t(x*y)$ where $h^t$ satisfying 
\begin{equation} \label{11}
\bold{h^t(x,1)=h^t(1,x)=h^t(x,x)=1}
\end{equation}
Now consider the expression,\\
$(ht(x))*(kt(y))=^k{(ht(x)*t(y))}=^{kh}(t(x)*t(y))
=kh \sigma_{(x*y)}(h^{-1}k^{-1})h^t(x,y)t(x*y)$

Thus the multiplicative Lie algebra structure $*$ on $G$ is defined by 
\begin{center}
 $\mathbf{(ht(x))*(kt(y))=hk \sigma_{(x*y)}(h^{-1}k^{-1})h^t(x,y)t(x*y)} ... (17)$
\end{center}
Now we will see properties of the function $h^t$ by using equations (16) and (17):

Consider the expression
$(ht(x))*((kt(y))\cdot (lt(z)))=(ht(x)) * (k\sigma_y(l)f^t(y,z)t(yz))=hk \sigma_{y}(l)f^t(y,z)\sigma_{x*(yz)}(h^{-1}k^{-1} \sigma_y(l^{-1})f^t(y,z)^{-1})h^t(x,yz)t(x*yz) ... (18)$

On the other hand
$$(ht(x))*((kt(y))\cdot (lt(z)))=(ht(x)*kt(y))\cdot ^{kt(y)}{(ht(x)*lt(z)})
=(hk\sigma_{x*y}(h^{-1}k^{-1})$$ $$h^t(x,y) t(x*y))(kt(y)(hl\sigma_{(x*z)}(h^{-1}l^{-1})h^t(x,z))t(x*z)t(y)^{-1}k^{-1})=(hk\sigma_{x*y}(h^{-1}$$ $$ k^{-1})h^t(x,y)t(x*y))(k\sigma_y(hl \sigma_{(x*z)}(h^{-1}l^{-1})h^t(x,z))t(y)t(x*z)f^t(y^{-1},y)^{-1}t(y^{-1})k^{-1})
 $$ $$=(hk\sigma_{x*y}(h^{-1}k^{-1})h^t(x,y)t(x*y))\cdot (k\sigma_y(hl \sigma_{(x*z)}(h^{-1}l^{-1})h^t(x,z))f^t(y,x*z)t(y(x*z))$$ 
$$f^t(y^{-1},y)^{-1}\sigma_{y^{-1}}(k^{-1})t(y^{-1})=hk\sigma_{x*y}(h^{-1}k^{-1})h^t(x,y) t(x*y)(k\sigma_y(hl \sigma_{(x*z)}(h^{-1}$$ $$l^{-1})h^t(x,z))f^t(y,x*z)\sigma_{y(x*z)}(f^t(y^{-1},y)^{-1}\sigma_{y^{-1}} (k^{-1}))f^t(y(x*z),y^{-1})t(^y(x*z))$$ $$=hk\sigma_{x*y}(h^{-1}k^{-1})h^t(x,y)\sigma_{x*y}(k\sigma_y^t(hl\sigma_{x*z}(h^{-1}l^{-1})h^t(x,z))f^t(y,(x*z))\sigma_{y(x*z)}($$ $$f^t(y^{-1},y)^{-1})\sigma_{y^{-1}}(k^{-1}))f^t(y(x*z),y^{-1}))f^t(x*y,^y(x*z))t((x*y)^y(x*z))
= hkh^t(x,y)$$ $$ \sigma_{x*y}(h^{-1}\sigma_y(hl  \sigma_{(x*z)}(h^{-1}l^{-1})h^t(x,z))f^t(y,x*z)\sigma_{y(x*z)}(f^t(y^{-1},y)^{-1}\sigma_{y^{-1}}(k^{-1}))$$ $f^t(y(x*z),y^{-1}))f^t(x*y,^y(x*z))t(x*(yz)) ... (19)$

By equating equation (18) and (19), we have
\begin{equation}\label{12}
\begin{split}
\bold{\sigma_{x*y}(h^{-1}\sigma_y(hl \sigma_{(x*z)}(h^{-1}l^{-1})}\bold{h^t(x,z))f^t(y,x*z)\sigma_{y(x*z)} (f^t(y^{-1},y)^{-1}\sigma_{y^{-1}}(k^{-1}))} \\ \bold{f^t(y(x*z),y^{-1})f^t(x*y,^y(x*z))h^t(x,y)= \sigma_{y}(l)f^t(y,z)\sigma_{x*(yz)}(h^{-1}k^{-1}\sigma_y(l^{-1})} \\ \bold{f^t(y,z)^{-1})h^t(x,yz)}
\end{split}
\end{equation}
Now consider the expression\\
$((ht(x))\cdot(kt(y)))*(lt(z))=hl \sigma_{x}(k)f^t(x,y) \sigma_{xy*z}(h^{-1}l^{-1}\sigma_x(k^{-1})f^t(x,y)^{-1})h^t(xy,z)\\t((xy)*z) ... (20)$

On the other hand
$$((ht(x))\cdot(kt(y)))*(lt(z))=^{ht(x)}(kt(y)*lt(z))\cdot (ht(x)*lt(z))=ht(x)(kl\sigma_{(y*z)}(k^{-1}l^{-1})$$ $$h^t(y,z)t(y*z))t(x)^{-1}h^{-1}\cdot(hl\sigma_{(x*z}(h^{-1}l^{-1})h^t(x,z)t(x*z))
=h \sigma_x(kl\sigma_{(y*z)}(k^{-1}l^{-1})h^t(y,z))$$ $$t(x)t(y*z)t(x)^{-1}h^{-1} (hl\sigma_{(x*z}(h^{-1}l^{-1})h^t(x,z)t(x*z))=h \sigma_x(kl\sigma_{(y*z)}(k^{-1}l^{-1})h^t(y,z))$$ $$f^t(x,y*z)t(x(y*z))t(x)^{-1}h^{-1}(hl\sigma_{(x*z}(h^{-1}l^{-1})h^t(x,z)t(x*z))=h \sigma_x(kl\sigma_{(y*z)} (k^{-1}l^{-1})$$ $$h^t(y,z))f^t(x,y*z)t(x(y*z))f^t(x^{-1},x)^{-1}t(x^{-1})h^{-1}\cdot (hl\sigma_{(x*z}(h^{-1}l^{-1})h^t(x,z)t(x*z)$$
$$=h \sigma_x(kl\sigma_{(y*z)}(k^{-1}l^{-1})h^t(y,z))f^t(x,y*z)\sigma_{(x(y*z))}(f^t(x^{-1},x)^{-1})f^t(x(y*z),x^{-1})$$ $$t(^x(y*z))(l\sigma_{(x*z}(h^{-1}l^{-1})h^t(x,z)t(x*z))
=h \sigma_x(kl\sigma_{(y*z)}(k^{-1}l^{-1})h^t(y,z))f^t(x,y*z)$$ $\sigma_{(x(y*z))} (f^t(x^{-1},x)^{-1})f^t(x(y*z),x^{-1})\sigma_{(^x(y*z))}(l\sigma_{(x*z}(h^{-1}l^{-1})h^t(x,z))t(^x(y*z))t(x*z))$

Therefore

$((ht(x))\cdot(kt(y)))*(lt(z))
= h \sigma_x^t(kl\sigma_{(y*z)}(k^{-1}l^{-1})h^t(y,z))f^t(x,y*z)\sigma_{(x(y*z))} (f^t(x^{-1},x)^{-1})\\ f^t(x(y*z),x^{-1})\sigma_{^x(y*z)} (l\sigma_{(x*z}(h^{-1}l^{-1})h^t(x,z))f^t(^x(y*z),(x*z))t(^x(y*z)(x*z))) ... (21)$

Further, equating the equations (20) and (21), we have
\begin{equation}\label{13}
\begin{split}
\bold{lf^t(x,y) \sigma_{(xy)*z}(h^{-1}l^{-1}\sigma_x(k^{-1})f^t(x,y)^{-1})h^t(xy,z)=
\sigma_x(l\sigma_{(y*z)}(k^{-1}l^{-1})h^t(y,z))} \\ \bold{f^t(x,y*z)\sigma_{x(y*z)}(f^t(x^{-1},x)^{-1})f^t(x(y*z),x^{-1})\sigma_{^x(y*z)}(l\sigma_{(x*z}(h^{-1}l^{-1})h^t(x,z))}\\ \bold{f^t(^x(y*z),(x*z))}
\end{split}
\end{equation}
Now consider the expression

 $$^{lt(z)}(ht(x)*kt(y))=lt(z)(hk\sigma_{(x*y)}(h^{-1}k^{-1})h^t(x,y)t(x*y))t(z)^{-1}l^{-1} =l\sigma_{z}(hk $$ $$\sigma_{(x*y)}(h^{-1}k^{-1})h^t(x,y))f^t(z,x*y)t(z(x*y))f^t(z^{-1},z)^{-1}t(z^{-1})l^{-1}
 =l\sigma_{z}(hk \sigma_{(x*y)}(h^{-1}$$ $k^{-1})h^t(x,y))f^t(z,x*y)\sigma_{(z(x*y)}(f^t(z^{-1},z)^{-1} \sigma_{z^{-1}}(l^{-1}))f^t(z(x*y),z^{-1})t(^z(x*y)) ... (22)$

On the other hand 
 $$^{lt(z)}(ht(x))*^{lt(z)}(kt(y))=(lt(z)ht(x)t(z)^{-1}l^{-1})*(lt(z)kt(y)t(z)^{-1}l^{-1})
=(l\sigma_z(h)$$ $$f^t(z,x)\sigma_{zx}(f^t(z^{-1},z)^{-1})t(zx)t(z^{-1})l^{-1})*(l\sigma_z (k)f^t(z,y)\sigma_{zy}(f^t(z^{-1},z)^{-1})t(zy)$$ $$t(z^{-1})l^{-1})=l^2\sigma_z(hk)f^t(z,x)f^t(z,y)\sigma_{zx} (f^t(z^{-1},z)^{-1})\sigma_{zy}(f^t(z^{-1},z)^{-1})\sigma_{^zx}(l^{-1})$$ $$\sigma_{^zy} (l^{-1})f^t(zx,z^{-1})f^t(zy,z^{-1})\sigma_{(^zx*^zy)}(l^{-2}\sigma_z(h^{-1}k^{-1})f^t(z,x)^{-1}f^t(z,y)^{-1}\sigma_{zx} ($$ $f^t(z^{-1},z))\sigma_{zy}(f^t(z^{-1},z))f^t(zy,z^{-1})^{-1}f^t(zx,z^{-1})^{-1}\sigma_{^zy}(l)\sigma_{^zx}(l))h^t(^zx,^zy)t(^zx*^zy) \\ ... (23)$

By equating the equation $(22)$ and $(23)$, we have
\begin{equation}\label{14}
\begin{split}
\bold{\sigma_{z}(\sigma_{(x*y)}(h^{-1}k^{-1})h^t(x,y))f^t(z,x*y)\sigma_{(z(x*y)}(f^t(z^{-1},z)^{-1}\sigma_{z^{-1}}(l^{-1}))}\\ \bold{f^t(z(x*y),z^{-1})=lf^t(z,x)f^t(z,y)\sigma_{zx}(f^t(z^{-1},z)^{-1})\sigma_{zy}(f^t(z^{-1},z)^{-1})\sigma_{^zx}(l^{-1})}\\   \bold{ \sigma_{^zy}(l^{-1})f^t(zy,z^{-1})f^t(zx,z^{-1}) \sigma_{(^zx*^zy)}(l^{-2}\sigma_z(h^{-1}k^{-1})f^t(z,x)^{-1}f^t(z,y)^{-1}} \\ \bold{\sigma_{zx} (f^t(z^{-1},z))\sigma_{zy}(f^t(z^{-1},z))f^t(zy,z^{-1})^{-1}f^t(zx,z^{-1})^{-1}\sigma_{^zy}(l)\sigma_{^zx}(l))h^t(^zx,^zy)}
\end{split}
\end{equation}  
  
By Jacobi identity, we have
\begin{equation} \label{15}
\begin{split}
\ \bold{hk^2\sigma_{x*y}(h^{-1}k^{-1})h^t(x,y)\sigma_y(l)f^t(y,z)\sigma_{yz} (f^t(y^{-1},y)^{-1}) f^t(yz,y^{-1})\sigma_{^yz}(k^{-1})} \\ \bold{\sigma_{(x*y)*^yz} (h^{-1}k^{-2}\sigma_{x*y} (hk)h^t(x,y)^{-1}\sigma_{y}(l^{-1})f^t(y,z)^{-1}\sigma_{yz}(f^t(y^{-1},y))f^t(yz,y^{-1})^ {-1}}\\ \bold{\sigma_{^yz}(k))h^t(x*y,^yz)\sigma_{(x*y)*^yz}(kl^2\sigma_{y*z}(k^{-1}l^{-1})h^t(y,z)
 \sigma_{z}(h)f^t(z,x)\sigma_{zx}(f^t(z^{-1},z))^{-1}}\\ \bold{f^t(zx,z^{-1})\sigma_{^zx}(l^{-1}) \sigma_{(y*z)*^zx}(k^{-1}l^{-2} \sigma_{y*z}(kl)h^t(y,z)^{-1}\sigma_{z}(h^{-1})f^t(z,x)^{-1}} \\ \bold{\sigma_{zx}(f^t(z^{-1},z))f^t(zx,z^{-1})\sigma_{^zx} (l))h^t(y*z,^zx))f^t((x*y)*^yz,(y*z)*^zx)} \\ \bold{\sigma_{((x*y)*^yz)*((y*z)*^zx)}(lh^2\sigma_{(z*x)}(l^{-1}h^{-1})h^t(z,x)\sigma_x(k)f^t(x,y)\sigma_{xy}(f^t(x^{-1},x)^{-1})} \\ \bold{f^t(xy,x^{-1})\sigma_{^xy}(h^{-1})\sigma_{((z*x)*^xy)}(l^{-1}h^{-2}\sigma_{z*x}(lh) h^t(z,x)^{-1}\sigma_x(k^{-1})f^t(x,y)^{-1}} \\ \bold{\sigma_{xy}(f^t(x^{-1},x))
  f^t(xy,x^{-1})^{-1}\sigma_{^xy}(h))h^t(z*x,^xy))} \\ \bold{f^t(((x*y)*^yz)*((y*z)*^zx),(z*x)*^xy)=1}
  \end{split}
\end{equation}
 \begin{definition} \label{Lie center factor system}
Let $H$ be an abelian group with trivial multiplicative Lie algebra structure and $K$ be an arbitrary multiplicative Lie algebra. Then a \textbf{Lie center factor system} is a quintuple $(K,H,f,h,\sigma)$, where $\sigma : K \longrightarrow Aut(H)$ is a group automorphism and $f, \ h$ are maps from $K \times K$ to $H$ satisfying the conditions like equations $(\ref{10}), (\ref{11}), (\ref{12}), (\ref{12}), (\ref{14})$ and $(\ref{15})$.
\end{definition}
\begin{proposition} \label{Lie center factor and extension}
Let $E(H, K)$ be Lie center extension of $H$ by $K$ with a choice of section $t$. Then their exists a Lie center factor system $\text{LFac}(E,t)=(K,H,f^t,h^t,\sigma)$. Conversely, for every Lie center factor system $(K,H,f,h,\sigma)$ we have a Lie center extension $E(H, K)$ of $H$ by $K$, with a section $t$ such that $\text{LFac}(E,t)=(K,H,f,h,\sigma)$.
\end{proposition}
\begin{proof}
From the above discussions and the Definition \ref{Lie center factor system}, it can be seen that every Lie center extension $E(H, K)$ of $H$ by $K$ with a given section $t$ determines a Lie center factor system $LFac(E, t) \equiv (K,H,f^t,h^t,\sigma).$ (This is called as a Lie center factor system given by the center extension E.)

Conversely let $(K,H,f,h, \sigma)$ be a Lie center factor system. Take $G = H \times K$ and define the $\cdot$ and $*$ binary operations on $G$ as
\begin{center}
$(a,x)\cdot (b,y)=(a\sigma_x (b)f(x,y),xy)$
and $(a,x)*(b,y)=(ab\sigma_{x*y}(a^{-1}b^{-1})h(x,y),x*y)$
\end{center}
It is easy to see that $(G,\cdot,*)$ is a multiplicative Lie Algebra such that
\begin{center}
\begin{tikzcd}
   1\arrow{r} & H\arrow{r}{i} & G \arrow{r}{p} & K\arrow{r} & 1
\end{tikzcd}
\end{center}
is a Lie center extension of $H$ by $K$ , where $i$ is the first inclusion and $p$ is the second projection. Let t be a section of $E(H, K)$ given by,
\begin{center}
$t(x)=(1,x)$
\end{center}
By an easy computation, it can be seen that $\sigma_E = \sigma, f^t=f~ \text{and}~ h^t=h.$
\end{proof}
  \subsection*{Equivalence between Category LEXT of Lie center extensions and Category LFAC of Lie center factor systems}
Let $(\lambda, \mu, \nu)$ be a morphism between the Lie center extensions $E_1(H, K)$ and $E_2(H, K)$ of $H$ by $K$ given by the following commutative diagram:
\begin{center}
\begin{tikzcd}
  E_1(H,K) \equiv  1\arrow{r} & H_1 \arrow{r}{i_1}\arrow{d}{\lambda} & G_1 \arrow{r}{\beta_1}\arrow{d}{\mu} & K_1\arrow{r}\arrow{d}{\nu} & 1 \\
   E_2(H,K) \equiv  1\arrow{r} & H_2\arrow{r}{i_2} & G_2\arrow{r}{\beta_2} & K_2\arrow{r} & 1
\end{tikzcd}
\end{center}
Let $t_1$ and $t_2$ be sections of $E(H_1, K_1)$ and $E(H_2, K_2)$ respectively. Consider the corresponding factor systems $(K_1,H_1,f^{t_1},h^{t_1}, \sigma^{1})$ and $(K_2,H_2,f^{t_2},h^{t_2}, \sigma^{2})$ of $E(H_1, K_1)$ and $E(H_2, K_2)$ respectively. Let $x \in K_1$. Then $\mu(t_1(x)) \in G_2$ and $\beta_2(\mu(t_1(x)))$ = $\nu(\beta_1(t_1(x)))$ = $\nu(x)$ = $\beta_2(t_2(\nu(x)))$. Thus, $\mu(t_1(x))(t_2(\nu(x)))^{-1}$ $ \in ker \beta_2 (= H_2)$. In turn, we have a unique $g(x)$ $\in H_2$ such that 
\begin{center}
$\mu(t_1(x))$ = $g(x)t_2(\nu(x)) ... (24)$ 
\end{center}
Since, $t_1(1)=1$ = $t_2(1)$,  it follows that
\begin{equation} \label{16}
\bold{g(1)=1}
\end{equation} 
Also we get the following equations  (see pages 375-376 (\cite{RL}))
\begin{equation} \label{17}
\bold{\lambda(f^{t_1}(x,y))g(xy)= g(x)\sigma_{\nu(x)}^{2}(g(y))f^{t_2}(\nu(x),\nu(y))}
\end{equation}
and
\begin{equation} \label{18}
\bold{g(x)\sigma_{\nu(x)}^{2}(\lambda(h))g(x)^{-1}=\lambda(\sigma_x^{1}(h))}
\end{equation}
Further,

 $\mu(t_1(x)*t_1(y))=\mu(h^{t_1}(x,y)t_1(x*y))
 = \lambda(h^{t_1}(x,y)) \mu(t_1(x*y))=\lambda(h^{t_1}(x,y))g(x*y)t_2(\nu(x)*\nu(y)) ... (25)$
 
On the other hand,
$\mu(t_1(x)*t_1(y))=\mu(t_1(x))*\mu(t_2(y))=(g(x)t_2(\nu(x)))*(g(y)t_2(\nu(y))$
$=g(x)g(y)\sigma_{(\nu(x)*\nu(y))}^2(g(x)^{-1}g(y)^{-1})h^{t_2}(\nu(x),\nu(y))t_2(\nu(x)*\nu(y)) ... (26)$

Now comparing equation $(25)$ and $(26)$, we have
\begin{equation} \label{19}
\bold{\lambda(h^{t_1}(x,y))g(x*y)=g(x)g(y)\sigma_{(\nu(x)*\nu(y))}^2(g(x)^{-1}g(y)^{-1})h^{t_2}(\nu(x),\nu(y))}
\end{equation}
Thus a morphism $(\lambda,\mu,\nu)$ between Lie center extensions $E_1(H,K)$ and $E_2(H,K)$ together with choices of sections $t_1$ and $t_2$ of the corresponding center extensions, induces a map $g$ from $K_1$ to $H_2$  such that the triple $(\nu, g, \sigma)$ satisfies the equations $(\ref{16}),(\ref{17}),(\ref{18})$ and $(\ref{19})$. This can be regarded as a morphism between the corresponding Lie center factor systems.

Now we introduce the category \textbf{LFAC} of Lie center factor systems whose objects are Lie center factor system, and a morphism from $(K_1,H_1,f^{1},h^{1}, \sigma^1)$ to $(K_2,H_2,f^{2},h^{2}, \sigma^2)$ is a triple $(\nu, g, \lambda)$, where $\nu : K_1 \longrightarrow K_2$, $\lambda: H_1 \longrightarrow H_2$ are multiplicative Lie algebra homomorphism, and $g: K_1 \longrightarrow H_2$ is a map such that 
\begin{enumerate}
\item $g(1) = 1$
\item $\lambda(f^{1}(x,y))g(xy)$ = $ g(x)\sigma_{\nu(x)}^{2}(g(y))f^{2}(\nu(x),\nu(y))$
\item $\lambda(\sigma_x^{1}(h))=g(x)\sigma_{\nu(x)}^{2}(\lambda(h))g(x)^{-1}$
\item $\lambda(h^{1}(x,y))g(x*y)=g(x)g(y)\sigma_{(\nu(x)*\nu(y))}^2(g(x)^{-1}g(y)^{-1})h^{2}(\nu(x),\nu(y))$
\end{enumerate}
The composition of morphisms $(\nu_1, g_1, \lambda_1)$ from $(K_1,H_1,f^{1},h^{1}, \sigma^1)$ to $(K_2,H_2,f^{2},h^{2}, \\\sigma^2)$ and $(\nu_2, g_2, \lambda_2)$ from $(K_2,H_2,f^{2},h^{2}, \sigma^2)$ to $(K_3,H_3,f^{3},h^{3}, \sigma^3)$ is the triple $(\nu_2 ~ o~ \nu_1,\\ g_3, \lambda_2 ~ o~ \lambda_1)$, where $g_3$ is given by $g_3 (x) = g_2(\nu_1 (x))\lambda_2(g_1(x))$ for each $x \in K_1$.

So, finally from above discussion we have the following theorem:
\begin{theorem} \label{equivalence between Lie}
There is an equivalence between the category $\mathbf{LEXT}$ of Lie center extensions to the category $\mathbf{LFAC}$ of Lie center factor systems.
\end{theorem}
\begin{theorem} \label{surjective map 1}
Let $LExt(H, K)$ denote the set of all equivalence classes of Lie center extensions of $H$ by $K$ and $Hom(K, Aut(H))$ denote the set of all group homomorphism from $K$ to $Aut(H)$. Then there is a natural surjective map $\eta : LExt(H, K) \rightarrow Hom(K, Aut(H))$ defined by $\eta([E]) = \sigma_{[E]}$.
\end{theorem}
\begin{proof}
From the Remark 3.1, the map $\eta :LExt(H, K) \longrightarrow Hom(K, Aut(H))$ defined by $\eta([E]) = \sigma_{[E]}$ is well defined. 
 
Let $\sigma : K\longrightarrow Aut(H)$ is a group homomorphism defined by $\sigma(x) = \sigma_x$. Let $G = H \times K$. Then $G$ is a multiplicative Lie algebra with the operations $\cdot $ and $*$, defined by $(h, x)\cdot (k, y) = (h\sigma_x(k),xy)$ and $(h, x)*(k, y) = (hk\sigma_{x*y}(h^{-1}k^{-1}), x*y)$ respectively. Consider a section $t : K \longrightarrow G$ is defined by $t(x) = (1, x)$. Therefore we have a Lie center extension $E(H, K)$ of multiplicative Lie algebra of $H$ by $K$
\begin{center}
\begin{tikzcd}
   1\arrow{r} & H\arrow{r}{\alpha} & G \arrow{r}{\beta} & K\arrow{r} & 1
\end{tikzcd}
\end{center}
such that $f^t$ and $h^t$ are trivial maps and $\sigma_E = \sigma$.
Thus every group homomorphism from $K\longrightarrow Aut(H)$ gives a Lie center extension of multiplicative Lie algebra of $H$ by $K$. This shows that $\eta$ is surjective.
\end{proof}
\begin{example}
The map $\eta$ in Theorem $\ref{surjective map 1}$ need not be injective. Specifically, there may exists same group homomorphism  $\sigma$ from $K$ to $Aut(H)$ for two distinct classes of extensions. Consider the center extensions:
\begin{center}
\begin{tikzcd}
E_1(\mathbb{Z},\mathbb{Z}_5) \equiv   \lbrace 0\rbrace \arrow{r} & \mathbb{Z}\arrow{r}{(i_1,o)} & \mathbb{Z}\bigoplus \mathbb{Z}_5 \arrow{r}{p_2} & \mathbb{Z}_5\arrow{r} &\lbrace 0\rbrace
\end{tikzcd}
\end{center}
and
\begin{center}
\begin{tikzcd}
E_2(\mathbb{Z},\mathbb{Z}_5) \equiv   \lbrace 0\rbrace\arrow{r} & \mathbb{Z}\arrow{r}{\alpha} & \mathbb{Z} \arrow{r}{\nu} & \mathbb{Z}_5\arrow{r} & \lbrace 0\rbrace
\end{tikzcd}
\end{center}
of the multiplicative Lie algebra $\mathbb{Z}$ by $\mathbb{Z}_5$, where $\alpha(n)=5n, ~ \forall ~ n \in \mathbb{Z}$ and $\nu$ is natural quotient map. Since $\mathbb{Z}$ and $\mathbb{Z}_5$ are cyclic groups, the multiplicative Lie algebra structure on these groups is trivial. Also it can be seen that $E_1(\mathbb{Z},\mathbb{Z}_5)$ and $E_2(\mathbb{Z},\mathbb{Z}_5)$ are not equivalent as center extension. Since $Aut(\mathbb{Z})$ is isomorphic to $\mathbb{Z}_2$, $Hom(\mathbb{Z}_5, Aut(\mathbb{Z}))$ is trivial. Therefore here it can be seen that $[E_1]\neq[E_2]$ but $\eta([E_1])=\eta([E_2])$.
\end{example}
Now we discuss the following problem:

\noindent\textbf{Problem:} Let $H$ be an abelian group with trivial multiplicative Lie algebra structure. Classify all Lie center extensions of $H$ by $K$ (up to equivalence) with the given group homomorphism $\sigma : K \longrightarrow Aut(H)$
\begin{definition} 
A multiplicative Lie center 2-co-cycle of a multiplicative Lie algebra K with coefficient in an abelian group H with trivial multiplicative Lie algebra structure is a triplet $(f,h,\sigma)$, where $f \in Z^2(K,H)$ is a 2-group co-cycle of K with coefficient in the trivial $K$-module $H$, $h:K \times K \longrightarrow H$ and $ \sigma : K \longrightarrow Aut(H)$  are the maps satisfy condition like equations $(\ref{10}),(\ref{11}),(\ref{12}),(\ref{13})$,$(\ref{14})$, and $(\ref{15})$. 
\end{definition} 
\begin{remark}
Let $E(H, K)$ be a center extension with a choice of section $t$. Then from above discussion, it is clear that we have a multiplicative Lie center 2-cocycle $(f^t,h^t,\sigma_E)$.

Conversely, if we have a multiplicative Lie center 2-cocycle  $(f,h, \sigma)$ of multiplication Lie Algebra $K$ with coefficient in the abelian group $H$ with trivial multiplicative Lie algebra structure, then from Proposition $(\ref{Lie center factor and extension})$ it can be seen that there exists a multiplicative Lie algebra $G$ and a center extension $E(H, K)$ with a choice of section $t$ such that $f^t=f$, $\sigma_{E}=\sigma$ and $h^t=h$.
\end{remark}
Let $E(H, K)$ be a Lie center extension with a choice of section $t$. Then we have a  multiplicative Lie center 2-cocycle $(f^t,h^t,\sigma_E)$.
Now, let s be an another section of $E$. Then their exists a map $g: K\rightarrow H$ with $g(e)=1$ such that $s(x)=g(x)t(x)$ for all $ x \in K$. Then

$$(hs(x))\cdot (ks(y))=h\sigma_x(k)f^s(x,y)s(xy)= h\sigma_x(k)f^s(x,y)g(xy)t(xy) ... (27)$$
On the other hand,

$(hs(x))\cdot (ks(y)) = h\sigma_x(k)s(x)s(y)=h\sigma_x(k)g(x)t(x)g(y)t(y)=h\sigma_x(k)g(x)\sigma_x(g(y))\\ f^t(x,y)t(xy) ... (28)$

Therefore on comparing equation $(27)$ and $(28)$ we have 
\begin{equation}
\bold{f^s(x,y)=g(x)g(xy)^{-1}\sigma_x(g(y))f^t(x,y)}
\end{equation}
Further

$(hs(x))*(ks(y))=hk\sigma_{x*y}(h^{-1}k^{-1})h^s(x,y)s(x*y)=hk\sigma_{x*y}(h^{-1}k^{-1})h^s(x,y)\\g(x*y)t(x*y) ... (29)$ 

Also we have
$$(hs(x))*(ks(y)) = hkg(x)g(y)\sigma_{x*y}(h^{-1}k^{-1}g(x)^{-1}g(y)^{-1})h^t(x,y)t(x*y) ... (30)$$

Thus on comparing equations $(29)$ and $(30)$ we have
\begin{equation}
\bold{h^s(x,y)=g(x)g(x*y)^{-1}g(y)\sigma_{x*y}(g(x)^{-1}g(y)^{-1})h^t(x,y)}
\end{equation}
These equations prompts us to the following definition:
\begin{definition}
Two multiplicative Lie center 2-cocycles $(f^s,h^s,\sigma)$ and $(f^t,h^t,\sigma)$ are said to be equivalent if there is an identity preserving map $g:K \longrightarrow H$ satisfying:
\begin{enumerate}
\item $f^s(x,y)=g(x)g(xy)^{-1}\sigma_x(g(y))f^t(x,y)$
\item $h^s(x,y)=g(x*y)^{-1}g(x)g(y)\sigma_{x*y}(g(x)^{-1}g(y)^{-1})h^t(x,y)$
\end{enumerate}
\end{definition}
The set $Z^2_{ML(\sigma)}(K,H)$ of all multiplicative Lie center 2-cocycles of $K$ with coefficient in $H$ is easily seen to be an abelian group with respect to coordinate wise operation given by $(f,h, \sigma)\cdot(f^{'},h^{'}, \sigma)=(ff^{'},gg^{'}, \sigma)$. Given any identity preserving map $g:K \longrightarrow H$, the triplet $(\delta g,g^{*}, \sigma)$ is a member of $Z^2_{ML(\sigma)}(K,H)$, where $\delta g,g^{*}$ are maps from $K \times K$ to $H$ given by $\delta g((x,y))=g(x)g(xy)^{-1}\sigma_x(g(y))$ and $g^*(x,y)=g(x*y)^{-1}g(x)g(y)\sigma_{x*y}(g(x)^{-1}g(y)^{-1})$.

Let $MAP(K,H)$ denote the group of identity preserving map from K to H. So we have a homomorphism $\chi:MAP(K,H)\longrightarrow Z^2_{ML(\sigma)}(K,H)$ given by $\chi(g)=(\delta g,g^*,\sigma)$.

The image of $\chi$ is called the group of multiplicative Lie center 2-coboundaries of $K$ with coefficient in $H$ and it is denoted by $B^2_{ML(\sigma)}(K,H)$. The quotient group $\frac{Z^2_{ML(\sigma)}(K,H)}{B^2_{ML(\sigma)}(K,H)}$ is called the second Lie center cohomology of $K$ with coefficient in $H$ and it is denoted by $H^2_{ML(\sigma)}(K,H)$. In turn, we get the following exact sequence of abelian groups,
$ 1\rightarrow  Hom(K,H)\xrightarrow{i} MAP(K,H)\xrightarrow{\chi} Z^2_{ML(\sigma)}(K,H)\xrightarrow{\nu} H^2_{ML(\sigma)}(K,H)\rightarrow 1$, 
where $\nu$ is quotient map.  Now we have following theorem which proof is similar to Theorem \ref{Bijective correspondence}.
\begin{theorem}
Let $H$ be an abelian group with trivial multiplicative Lie algebra structure, $K$ be a multiplicative Lie algebra and $\sigma$ be a group homomorphism from $K$ to the group of automorphisms $Aut(H)$ of $H$. Then there is a bijective correspondence between the set $LExt_\sigma(H,K)$ of equivalence classes of Lie center extensions of $H$ by $K$ with the given $\sigma$ and the second Lie co-homology $H^2_{ML(\sigma)}(K,H)$.   
\end{theorem}
\begin{remark}
As in Section 2, we can define Baer sum on class of Lie center extensions and with respect to that Baer sum $LExt_\sigma(H,K)$ forms an abelian group which is isomorphic to $H^2_{ML(\sigma)}(K,H)$.
\end{remark}
\begin{remark} By Section 2 and Section 3, it is easy to see that for any extension \begin{tikzcd}
  E(H, K) \equiv 1\arrow{r} & H\arrow{r}{i} & G \arrow{r}{\beta} & K\arrow{r} & 1
\end{tikzcd}
of an abelian group $H$ with trivial multiplicative Lie algebra structure by an arbitrary multiplicative Lie algebra $K$, the group operation $``\cdot"$ and the multiplicative Lie algebra structure $``*"$ in $G$ are given by
$$(ht(x))\cdot (kt(y))=h\sigma_x^t(k)f^t(x,y)t(xy)$$
 
$$(ht(x))* (kt(y))=hk\Gamma^t_x(k)\sigma_{(x*y)}(h^{-1}k^{-1}\Gamma^t_y(h^{-1})) h^t(x,y)t(x*y)$$

where $\sigma_x^t(k)=t(x)kt(x)^{-1}$, $\Gamma^t_x(k)=t(x)*k$ are group homomorphism on $H$
and $f^t, h^t: K\times K \longrightarrow H$ are maps satisfying equations (\ref{10}) and (\ref{11}). One can discuss other properties of the extension $E(H, K)$ accordingly.
\end{remark}
\textbf{Acknowledgement}: We are extremely thankful to Prof. Ramji Lal for his valuable suggestions, discussions and constant support. The first named author thanks IIIT Allahabad for providing institute fellowship and the second named author also thanks IIIT Allahabad for providing one time seed money project..

\end{document}